\documentclass[oneside]{amsart}
\usepackage{amsmath,amssymb,amsfonts,amsthm}
\usepackage{color}
\usepackage{graphicx}
\usepackage{geometry} 
\usepackage{amscd}
\usepackage{enumitem}
\usepackage{maplestd2e}

\title{\Large S\MakeLowercase{olutions for the} L\MakeLowercase{andsberg unicorn  problem in} F\MakeLowercase{insler geometry}}

\author[Elgendi]{\bf S. G. ~E\MakeLowercase{lgendi}}

\address{S.G.~Elgendi, Department of Mathematics, Faculty of Science, Benha
  University, Egypt, \phantom{mmmmm}
 Institute of Mathematics, University of Debrecen,
  Debrecen, Hungary} \email{salah.ali@fsci.bu.edu.eg, \ \ salahelgendi@yahoo.com.}

\keywords{Berwald metrics;  Landsberg metrics; $(\alpha,\beta)$-metrics, Finsler packages; Maple program.}

\subjclass[2010]{53C60, 53B40, 58B20, 68U05, 83-04, 83-08.}

\thanks{}

\def\blue#1{\textcolor[rgb]{0.0,0.0,1.0}{#1}}

\newcommand{\T}{{\mathcal T}}
\newcommand{\C}{{\mathcal C}}

\newcommand{\Real}{\mathbb R}

\newcommand{\set}[1]{\left\{#1\right\}}

\newcommand{\To}{\longrightarrow}

\newcommand{\tm}{\T M}

\def\paa{\dot{\partial}}

\def\+{\!+\!}

\def\={\!=\!}
\def\<{\!<\!}
\def\>{\!>\!}

 \let\oldmarginpar\marginpar
\renewcommand\marginpar[1]{\oldmarginpar[\raggedleft\footnotesize #1]%
  {\blue{\raggedright \footnotesize \fbox{
      \begin{minipage}{1.0\linewidth}
        #1
      \end{minipage}
}}}}

\numberwithin{equation}{section} 
\numberwithin{figure}{section} 

\theoremstyle{plain}
\newtheorem*{theorem*}{Theorem}
\newtheorem{theorem}{Theorem}[section]

\theoremstyle{definition}
\newtheorem{definition}[theorem]{Definition}

\newtheorem{example}[theorem]{Example}
\newtheorem{remark}[theorem]{Remark}
\newtheorem*{acknowledgement*}{Acknowledgement}

\newcommand\undersym[2]{\raisebox{-7pt}{\tiny$#2$}{\kern-8pt}\mbox{$#1$}}
\newcommand\undersymm[2]{\raisebox{-7pt}{\tiny$#2$}{\kern-15pt}\mbox{$#1$}}

\begin{document}

\maketitle

\bigskip

\begin{center}
\textit{Dedicated to the memory of Professor Lajos Tam\'{a}ssy
}\end{center}

\bigskip

\begin{abstract}
 It is still a long-standing open problem in Finsler geometry, is there any regular Landsberg metric which is not Berwaldian. However, there are non-regular Landsberg metrics which are not Berwladian. The known examples are established by G. S. Asanov and Z. Shen. In this paper, we use the Maple program to study some explicit examples of non-Berwaldian Landsberg metrics. In fact, such kinds of examples are very tedious and complicated to investigate. Nonetheless, we use the maple program and Finsler packages to simplify the calculations in an elegant way. Depending on these examples, we manage to figure out some geometric properties of the geodesic spray of a non-Berwaldain Landsberg metric. Deforming this spray in a very specific way, using the metrizability tools of the deformed spray, we get new (very simple) non-Berwaldian Landsberg metrics. Moreover, the powerful  of this procedure is investigating a very simple and useful formula for the general class obtained by Z. Shen.
 \end{abstract}

\section{Introduction}

The examples and applications of    Finsler geometry, principally,  in mathematics and physics,  are very tedious and intricate  to accomplish.  For this reason, one of the benefits of using a computer is to manipulate the complicated calculations with saving time and efforts.   Moreover, this  enables  to study many  examples  in various applications  (cf., for example, \cite{r101}, \cite{Rutz2},  \cite{Shen-book}, \cite{r93},\cite{Portugal1}). The Finsler package \cite{Rutz3} included in \cite{hbfinsler1} and the new Finsler package \cite{NF_Package} are   good  illustrations of using computer in the applications of Finsler geometry.

\medskip

A Finsler manifold $(M, F)$ is said be  \textit{Berwald } if the coefficients of  Berwald connection depend only on the position arguments and this is equivalent to that the geodesic spray of $F$ is quadratic. A Finsler manifold $(M, F)$  is said be \textit{Landsberg } if the horizontal covariant derivative of the metric tensor of $F$ with respect to Berwald connection vanishes. Lots of  characterizations for both Berwald and Landsberg metrics can be found in the literature.

It is known that every Berwald space is  Landsberg. Whether there are Landsberg spaces which are not of Berwald type is a long-standing question in Finsler geometry, which is still open.  Despite the extraordinary  effort by many Finsler geometers, it is not known an example of a regular non-Berwaldian Landsberg space.

\medskip

 In \cite{Asanov}, G. S. Asanov  obtained  examples, arising from Finslerian General Relativity, of non-Berwaldian Landsberg spaces, of dimension at least $3$.  In Asanov's examples the Finsler functions are not defined for all values of the fiber coordinates $y^i$ (non-regular). Whether or not there are regular non-Berwaldian Landsberg spaces remains an open question.   In \cite{Shen_example}, Z. Shen  studied the class of $(\alpha,\beta)$ metrics of Landsberg type, of which Asanov's examples are particular cases; he found \cite{Shen_etal,Shen_example} that  there are non-regular non-Berwaldian Landsberg spaces with $(\alpha,\beta)$ metrics, there are no regular ones. Bao \cite{Bao} tried to  construct non-Berwaldian Landsberg spaces by successive approximation,  but this method  so far could not  solve the problem. The elusiveness of regular non-Berwaldian Landsberg spaces leads Bao to describe them as the unicorns of Finsler geometry. Lots of papers studying this problem can be found in the literature, for example, we refer to \cite{Crampin} and the references therein.

\medskip

In this paper, using computer and Maple program, we investigate various fascinating  and very simple examples of Landsberg non-Berwaldian metrics. Precisely, we study some explicit examples of non-regular Landsberg metrics which are not Berwaldian. In fact, the calculations of such examples are not easy to do. However, we manage to use the Maple program to simplify the calculations to very simple  formulae. Depending on the simplified formulae we point out some remarks which are useful not only to  show some geometric meaning of a spray of Landsberg metric, but also  to find new examples.  The most interesting and useful remark  that we point out  is the following.

 According to Shen \cite{Shen_example}, the class of non Berwaldian Landsberg metrics  on a manifold $M$ has extreme directions.   Without loss of generality, we fix the extreme (singular) directions of the Landsberg metric $(M,F)$, which is investigated by Z. Shen, as $(\pm 1 , 0,...,0)$, then the coefficients of the geodesic spray are given by
$$G^1=f_{ij}(x^1)y^iy^j, \quad G^\mu=Py^\mu,$$
where $f_{ij}$ are some smooth functions on $M$. Hence, the Berwald tensor is given by
$$G^1_{ijk}=0,\quad G^h_{1jk}=0,\quad G^\mu_{\lambda \nu \gamma}=P_{\lambda \nu \gamma} \ y^\mu +P_{\lambda \nu } \ \delta^\mu_\gamma+P_{\nu \gamma  } \ \delta^\mu_\lambda+P_{\gamma\lambda  } \ \delta^\mu_\nu, $$
where $P$ is  defined and  smooth only on an open subset of $T M$, $P_i:=\dot{\partial}_i P, P_{ij}:=\dot{\partial}_i\dot{\partial}_j P, P_{ijk}:=\dot{\partial}_i\dot{\partial}_j \dot{\partial}_k P$ and moreover, $P_{1j}=0$. Throughout, the Greek letters $\lambda, \mu, \nu, \gamma$ run over  $2,...,n$ and $\dot{\partial}_i$ stands for the partial derivative with respect to directions $y^i$.
Therefore, the Landsberg tensor takes the form

$$
L_{\lambda \nu \gamma}=-\frac{1}{2}F(P_{\lambda \nu \gamma} \ \ell_\mu y^\mu +P_{\lambda \nu } \ \ell_\gamma+P_{\nu \gamma  } \ \ell_\lambda+P_{\gamma\lambda  } \ \ell_\nu), \quad \ell_i:=\paa_iF.$$
In this case, it is clear  that the Landsberg tensor is given in terms of the function $P$ and the Finsler function $F$  as well. This specific form of the spray makes the condition of  Finsler metric for  being Landsberg much weaker; it is possible to find a function $P$ satisfies $L_{\lambda \nu \gamma}=0$ instead of the existence of $n$ different  functions  $G^i$ such that $L_{ijk}=-\frac{1}{2}F\ell_h G^h_{ijk}=0$.

\medskip

The main aim of this paper is to provide  new solutions for the unicorn problem, for this reason,  we deform the spray of a non-Berwaldain Landsberg metric in  a very specific way. Then, studying the metrizability of the new sprays we get the solutions.
For $n\geq 3$, arbitrary constants $a,p,q$, and  $\beta$ and $\alpha$ are given in Subsection 4.1, we investigate the  following classes on an $n$-dimensional manifold $M$

$${F}=\left(a \beta+\sqrt{\alpha^2-\beta^2}\right)\,e^{\frac{a \beta}{a\beta+\sqrt{\alpha^2-\beta^2}}}, \quad a\neq 0,$$

$${F}=\left((a+1)\beta +\sqrt{\alpha^2-\beta^2}\right)^{(1+a)/2}\,\left((a-1) \beta+\sqrt{\alpha^2-\beta^2}\right)^{(1-a)/2}, \quad a\neq 0 ,\pm 1,$$

$${F}=f(x^1)\left(a\beta+\frac{\alpha^2-\beta^2}{a\beta+2\sqrt{\alpha^2-\beta^2}}\right), \quad a\neq 0,$$

$$F=\sqrt{\alpha^2+p\beta\sqrt{\alpha^2-\beta^2}+q\beta^2}\,e^{\frac{p}{\sqrt{p^2-4q-4}}\, \text{arctanh}\left(\frac{p\beta+2\sqrt{\alpha^2-\beta^2}}{\beta\sqrt{p^2-4q-4}}\right)}, \quad p\neq 0,$$
  are classes of Landsberg metrics which are not Berwaldian. For  further discussion about these classes, how they are new or related to the known ones, we pay reader's attention to the concluding remarks at the end of this paper.

\section{Preliminaries}

Let  $M$ be  an n-dimensional smooth manifold.    The tangent space to $M$ at $p$ is denoted by  $T_pM$; $TM:=\undersymm{\bigcup}{p\in M}\,\,T_{p}M$ is the tangent bundle of $M$, $\tau:TM\longrightarrow M$ is the tangent bundle projection. We fix a chart $(\mathcal{U}, (u^1,...,u^n))$ on $M$. It induces a local coordinate system $(x^1,...,x^n, y^1,...,y^n)$ on $TM$, where
$$x^i:=u^i\circ \tau, \quad y^i(v):=v(u^i) \quad (v\in \tau^{-1}(\mathcal{U})).$$
By abuse of notation, we shall denote the coordinate functions $u^i$ also by $x^i$.

 The vector $1$-form $J$ on $TM$ defined,
locally, by $J = \frac{\partial}{\partial y^i} \otimes dx^i$ is called the
natural almost-tangent structure of $T M$. The vertical vector field
$\C=y^i\frac{\partial}{\partial y^i}$ on $TM$ is called the canonical or the
Liouville vector field.

A vector field $S\in \mathfrak{X}(\T M)$ is called a spray if $JS = \C$ and
$[\C, S] = S$. Locally, a spray can be expressed as follows
\begin{equation}
  \label{eq:spray}
  S = y^i \frac{\partial}{\partial x^i} - 2G^i\frac{\partial}{\partial y^i},
\end{equation}
where the \emph{spray coefficients} $G^i=G^i(x,y)$ are $2$-homogeneous
functions in  $y$.

A nonlinear connection is defined by an $n$-dimensional distribution $H : z \in \tm \rightarrow H_z(\tm)\in T_z(\tm)$ that is supplementary to the vertical distribution, which means that for all $z \in \tm$, we have
$T_z(\tm) = H_z(\tm) \oplus V_z(\tm).$

Every spray S induces a canonical nonlinear connection through the corresponding horizontal and vertical projectors,
\begin{equation}
  \label{projectors}
    h=\frac{1}{2}  (Id + [J,S]), \,\,\,\,            v=\frac{1}{2}(Id - [J,S])
\end{equation}
 With respect to the induced nonlinear connection, a spray $S$ is horizontal, which means that $S = hS$. Locally, the two projectors $h$ and $v$ can be expressed as follows
$$h=\frac{\delta}{\delta x^i}\otimes dx^i, \quad\quad v=\frac{\partial}{\partial y^i}\otimes \delta y^i,$$
$$\frac{\delta}{\delta x^i}=\frac{\partial}{\partial x^i}-G^j_i(x,y)\frac{\partial}{\partial y^j},\quad \delta y^i=dy^i+G^j_i(x,y)dx^i, \quad G^j_i(x,y)=\frac{\partial G^j}{\partial y^i}.$$

If $X\in  \mathfrak{X}(M)$, $ i_{X}$ and $\mathcal{L}_X$ denote the interior product by $X$ and the Lie derivative  with respect to $X$, respectively. The differential of $f\in C^\infty(M) $ is $df$. A vector $\ell$-form on $M$ is a skew-symmetric $C^\infty(M)$-linear map $L:(\mathfrak{X}(M))^\ell\longrightarrow \mathfrak{X}(M)$. Every vector $\ell$-form $L$ defines two graded derivations $i_L$ and $d_L$ of the exterior algebra of $M$ such that
 $$i_Lf=0, \quad i_Ldf=df\circ L, \quad d_L:=[i_L,d]=i_L\circ d-(-1)^{\ell-1}di_L,\quad (f\in C^\infty(M)).$$

Throughout, we use the following notations for the partial differentiation with respect to the position arguments $x^i$ and the direction arguments $y^i$
 $${\partial}_i:=\frac{\partial}{\partial x^i}, \quad\dot{\partial}_h:=\frac{\partial}{\partial y^h}.$$

\begin{definition}
A Finsler manifold  of dimension $n$ is a pair $(M,F)$, where $M$ is a  differentiable manifold of dimension $n$ and $F$ is a map  $$F: TM \To \Real ,\vspace{-0.1cm}$$  such that{\em:}
 \begin{description}
    \item[(a)] $F$ is smooth and strictly positive on $\T M$ and $F(x,y)=0$ if and only if $y=0$,
    \item[(b)]$F$ is positively homogenous of degree $1$ in the direction argument $y${\em:}
    $\mathcal{L}_{\mathcal{C}} F=F$,
    \item[(c)] The metric tensor $g_{ij}= \dot{\partial}_i \dot{\partial}_j E$ is non-degenerate on $\T M$, where $E:=\frac{1}{2}F^2$ is the energy function.
 \end{description}
 In this case $(M,F)$ is called regular Finsler space. If $F$ is not smooth or even not defined in some directions, then we call the Finsler Function non-regular.
 \end{definition}

 Since the $2$-form $dd_JE$ is non-degenerate,  the Euler-Lagrange equation
\begin{equation*}
  i_Sdd_JE=-dE
\end{equation*}
uniquely determines a spray $S$ on $TM$.  This spray is called the
\emph{geodesic spray} of the Finsler function.

\begin{definition}
  A spray $S$ on a manifold $M$ is called \emph{Finsler metrizable} if there
  exists a Finsler function $F$ such that the geodesic spray of the Finsler
  manifold $(M,F)$ is $S$.
  \end{definition}

  It is known that a spray $S$ is Finsler metrizable if and only if there exists a non-degenerate  solution $F$    for the system
  \begin{equation}
  \label{metrizable_system}
  d_hF=0, \quad d_\C F=F,
  \end{equation}
  where $h$ is the horizontal projector associated to $S$.

For a Finsler metric $F = F(x, y)$ on  $M$, the coefficients $G^i$ of the geodesic spray $S $   are given  by
\begin{equation}
\label{geodesic_spray}
G^i= \frac{1}{4}g^{ih}( y^r\partial_r\dot{\partial}_hF^2 - \partial_hF^2).
\end{equation}
The  components $G^i_j$ of the non linear connection  and the coefficients $G^i_{jk}$ of Berwald connection  are defined, respectively, by
$$G^i_j=\dot{\partial}_jG^i, \quad G^i_{jk}=\dot{\partial}_kG^i_j.$$
 The Berwald tensor $G^i_{jkh}$ and  Landsberg tensor $L_{jkh}$  are defined,  respectively, by
  $$
  G^i_{jkh}=\dot{\partial}_hG^i_{jk},
 \quad
  L_{jkh}=-\frac{1}{2}F\ell_i G^i_{jkh},
  $$
 where $\ell_i:=\dot{\partial}_iF$ is  the  normalized supporting element.

  \begin{definition}
A Finsler manifold $(M,F)$ is said to be \textit{Berwald} if the Berwald tensor $G^{h}_{ijk}$ vanishes identically.
\end{definition}

  \begin{definition}
A Finsler manifold $(M,F)$  is called \textit{Landsberg} if the Landsberg tensor $L_{ijk}$ vanishes identically.
\end{definition}

It is clear that every Berwald space is  Landsberg. Whether there
are Landsberg spaces which are not of Berwald type is a long-standing question
in Finsler geometry, which is still open.  Despite the intense effort by many Finsler
geometers, it is not known an example of a regular  non Berwaldian Landsberg space. However, there are some classes of non-regular Landsberg metrics which are not Berwladian. These classes were obtained by G. S. Asanov and later were generalized by Z. Shen. Moreover, these classes  are special $(\alpha,\beta)$-metrics.

\medskip

Let $\alpha$ be a Riemannian metric, $\beta$ a 1-form on $M$. Locally,
$$\alpha \underset{(\mathcal{U})}{=} a_{ij}\,dx^i\otimes dx^j, \quad \beta \underset{(\mathcal{U})}{=} b_{i}\,dx^i.$$

The Riemannian metric $\alpha$ induces naturally a Finsler function $F_\alpha$ on $TM$ given by $F_\alpha(v):=\sqrt{\alpha_{\tau(v)}(v,v)}$. Similarly, the 1-form $\beta$ can be interpreted as a smooth function
$$\overline{\beta}:TM\longrightarrow \Real, \quad v\longmapsto \overline{\beta}(v):=\beta_{\tau(v)}(v).$$
Locally,
$$F_\alpha\underset{(\mathcal{U})}{=}\sqrt{(a_{ij}\circ\tau ) y^iy^j}, \quad \overline{\beta}\underset{(\mathcal{U})}{=}(b_i\circ \tau )y^i.$$
In what follows, as usual, we shall simply write $\alpha$ and $\beta$ instead of $F_\alpha$ and $\overline{\beta}$, respectively.

For any $p\in M$, we define
$$\| \beta_p\|_\alpha:=\sup_{v\in T_pM\backslash\{0_p\}}\frac{\beta(v)}{\alpha(v)}.$$

An  $(\alpha,\beta)$-metric for $M$ is a function $F$ on $\T M:=\undersymm{\bigcup}{x\in M}(T_xM\backslash\{0_x\})$ defined by
$$F:=\alpha \phi(s):=\alpha(\phi\circ s), \quad s:=\frac{\beta}{\alpha},$$
where $\phi :(-b_0,b_0)\longrightarrow \Real$ is a smooth function $(b_0>0)$.
For more details regarding the regularity or non-regularity of the $(\alpha,\beta)$-metrics, we refer to \cite{Shen_example}.

 \bigskip

 The geodesic spray of an  $(\alpha,\beta)$-metric is given by

 \begin{equation}
 \label{spray_alpha-beta_metric}
 G^i=G_\alpha^i+\alpha Q s^i_{\,\,0}+\Theta\set{-2\alpha Q s_0+r_{00}}\set{\frac{y^i}{\alpha}+\frac{Q'}{Q-sQ'}b^i},
 \end{equation}
 where
 $$r_{ij}=\frac{1}{2}(b_{i|j}+b_{j|i}), \quad s_{ij}=\frac{1}{2}(b_{i|j}-b_{j|i}),$$
 $$r_{00}:=r_{ij}y^iy^j, \quad s_{i0}:=s_{ij}y^j, \quad s^i_{\,\,j}:=s_{hj}a^{ih},\quad s_i:=s_{ij}b^j,\quad b^i:=b_ja^{ij}$$
 \begin{equation}
 \label{Q}
 Q(t):=\frac{\phi'(t)}{\phi(t)-t\phi'(t)}, \quad t\in (-b_0,b_0)
 \end{equation}
 \begin{equation}
 \label{Theta}
 \Theta(t):=\frac{Q(t)-tQ'(t)}{2(1+tQ(t)+(b^2-t^2)Q'(t))},
 \end{equation}
  the symbol  $|$  refers to the covariant derivative with respect to the Levi-Civita connection of $\alpha$ and $Q'$ (resp. $\phi'$) mean the derivative of $Q$ (resp. $\phi$) with respect to $t$.  In what follows,  we use the notations $Q(s):=Q\circ s$, $\Theta(s):=\Theta\circ s$.

  \bigskip
  The following class of $(\alpha,\beta)$-metrics is a class of Landsberg metrics which are not Berwaldian. This class has been obtained by Z. Shen \cite{Shen_example}.
  $$F=\alpha \phi(s),$$
\begin{equation}
\label{Shen_class}
\phi(s)=c_4\sqrt{1+s\left(c_1\sqrt{1-s^2}+c_3s\right)} \exp\left(\frac{c_1\arctan \psi}{\sqrt{(2+c_3)^2-(c_1^2+c_3^2)}}\right),
\end{equation}
where $\psi$ is given by any of the following formulas
$$\psi=\frac{\left(c_3 r +(2+c_3)(c_1+r)\right)s+\left( r (c_1+r)-(2+c_3)c_3\right)\sqrt{1-s^2}}{\left(c_3 s+(c_1+r )\sqrt{1-s^2}\right)\sqrt{(2+c_3^2)-r^2}},$$
$$\psi=\frac{\left((2+c_3)c_3+r(r-c_1)\right)s+\left(c_3r-(2+c_3)(r-c_1)\right)\sqrt{1-s^2}}{\left(c_3 \sqrt{1-s^2}+(r-c_1)s\right)\sqrt{(2+c_3^2)-r^2}},$$
where $r:=\sqrt{c_1^2+c_3^2}$,
 $c_1$, $c_3$, $c_4$ are constants with $c_1\neq 0$, $1+c_3b_0>0$, $c_4>0$ $k$ is a non zero scalar function on $M$. Moreover,
$$b=b_0,\quad b_{i|j}=b_{j|i}, \quad  b_{i|j}=k (a_{ij}-b_ib_j). $$

\medskip
Particularly,  when $c_3=0$ and $c_1=g$
\begin{equation}
\label{Asanov_class}
\phi(s)=c_4\sqrt{1+gs\sqrt{1-s^2}} \exp\left(\frac{g\arctan \psi}{\sqrt{4-g^2}}\right),
\end{equation}
where $\psi$, in this case,  is given by
\[\psi=  \left\{
\begin{array}{ll}
     \frac{2s+g\sqrt{1-s^2}}{\sqrt{4-g^2}\sqrt{1-s^2}}, & g> 0 \\
     &\\
      -\frac{gs+2\sqrt{1-s^2}}{s\sqrt{4-g^2}} , & g< 0 \\
\end{array}
\right.
\]
This class has been obtained by  G. S. Asanov \cite{Asanov}.
  These classes are almost regular Landsberg  metrics which are not  Berwladain and they are Berwaldain if and only if $k=0$.

\medskip

Z. Shen showed that the class of almost regular $(\alpha,\beta)$-metrics $F=\alpha \phi(\beta/\alpha)$ might be singular or even not defined in two extreme directions $y\in T_xM$ with $\beta(x,y)=\pm b_0\alpha(x,y)$.

On the other hand, in \cite{Elgendi-ST_condition}, the author introduced the concept of  $\sigma$T-condition; that is, a Finsler manifold $(M,F)$ satisfies the $\sigma$T-condition if   there is a function $\sigma(x)$ on $M$ such that
$$\sigma_hT^h_{ijk}=0, \quad \sigma_h:=\frac{\partial \sigma}{\partial x^h}, $$
where $T^h_{ijk}$ is the T-tensor. Moreover, it was proven that the  $(\alpha,\beta)$-metrics satisfy the $\sigma$T-condition
 if and only if these metrics are given by the  class \eqref{Shen_class}, in this case $\sigma_i=f(x)b_i$, for some function $f(x)$.

 By the help of \cite{Elgendi-LBp},  we can conclude that, without loss of generality, the non-regular examples of non-Berwaldian Landsberg metrics can be chosen in such a way  that it is a conformal transformation of a Berwald or Minkoweski metrics by a function $f(x^1)$ and hence $\beta=f(x^1) y^1$, i.e $b_1=f(x^1)$. Consequently,  the  directions of singularities of the metric will be $(\pm 1,0,...,0)$. So, if the Finsler function has extreme directions in the directions of  $ y^1$, say,  then  $\beta$ can  be in the form $\beta=f(x^1)y^1$. For this reason, in all of the following examples, we fix $\beta=f(x^1)y^1$ for smooth positive function $f(x^1)$ on the base manifold $M$.

As an illustration to the above paragraph, let's give the following simple case of an example given by Shen \cite{Shen_example}.
\begin{example}
Let $M= \mathbb{R}^3$, and   $\alpha=\sqrt{(y^1)^2+e^{2x^1}((y^2)^2+(y^3)^2)}$, $\beta=y^1$. Then, the Finsler function $F$   given by
\begin{equation}\label{Eq:1}
F=\sqrt{\alpha^2+\beta\sqrt{\alpha^2-\beta^2}}\,e^{\frac{1}{\sqrt{3}} \ \arctan\Big(\frac{2\beta}{\sqrt{3}\sqrt{\alpha^2-\beta^2}}+\frac{1}{\sqrt{3}}\Big)}.
\end{equation}
is a non-Berwaldain Landsberg metric.
\end{example}
One can obtain the example by applying the strategy mentioned in \cite{Elgendi-LBp}. Precisely, by the choice  $\alpha=\sqrt{e^{-2x^1}(x^1)^2+((y^2)^2+(y^3)^2)}$, $\beta=e^{-x^1} y^1$, the Finsler metric \eqref{Eq:1}, is Berwaldain. Now, applying the conformal transformation on $F$ by the function $e^{x^1}$ we will get the same metric. Moreover, making use of \cite{Elgendi-LBp}, the conformal transformation of $F$ by any positive smooth function $f(x^1)$ will yield a non-Berwadlian Landsberg metric.
\section{Some explicit examples and remarks}

We start this section, using the Maple program and NF-package, by studying some explicit examples of non-regular Landsberg metrics which are not Berwaldian. In fact, the calculations of such examples are not easy to do. However, we manage to use the Maple program to simplify the examples to very simple formulae (comparing to the formulae which are obtained by the NF-package). According to the simplified formulae, we point out some remarks. These remarks  are useful, not only to figure out some geometric properties of a geodesic spray of non-Berwaldian Landsberg metric but also  to establish  examples that have never been obtained.

\bigskip

From now on,  we fix the function $f(x^1)$ as a positive smooth function on $M$. The following examples,  can be obtained from \eqref{Shen_class} or \eqref{Asanov_class}. By choosing $\alpha=f(x^1)\sqrt{(y^1)^2+(y^2)^2+(y^3)^2}$ or $\alpha=f(x^1)\sqrt{(y^1)^2+y^2y^3}$ and $\beta=f(x^1)y^1$. As a verification, we propose a procedure based on Maple program and NF-package in the following  examples (the first example is mentioned in \cite{ND-Chern}).

\begin{example}\label{ex.1}
Let $M= \mathbb{R}^3$, and    $F$  be given by
$$F(x,y):=f(x^1)\sqrt{{(y^1)}^2+y^2y^3+y^1\sqrt{y^2y^3}}\,e^{\frac{1}{\sqrt{3}} \ \arctan\Big(\frac{2y^1}{\sqrt{3y^2y^3}}+\frac{1}{\sqrt{3}}\Big)}.$$
Using the NF-package, we can find the coefficients of the geodesic spray of $F$. But the formulae will be very complicated. One way to simplify these formulae is to copy them to a separate  worksheet and simplify them.  Then, we get the very simplified  coefficients of the geodesic spray of $F$ as follows:
$$G^1=\left((y^1)^2-y^2y^3\right)\frac{f'(x^1)}{2f(x^1)},\quad G^2=y^2\left(2y^1+\sqrt{y^2y^3}\right)\frac{f'(x^1)}{2f(x^1)},\quad G^3=y^3\left(2y^1+\sqrt{y^2y^3}\right)\frac{f'(x^1)}{2f(x^1)},$$
where $f'(x^1)$ is the derivative of $f(x^1)$ with respect to $x^1$.

To check if the Finsler function $F$ is Landsbergian and non-Berwaldian, we use Maple program and the following procedure.

\bigskip

\begin{maplegroup}
\begin{mapleinput}
\mapleinline{active}{1d}{restart}{\[{\it restart}\]}
\end{mapleinput}
\end{maplegroup}
\begin{maplegroup}
\begin{mapleinput}
\mapleinline{active}{1d}{F := f(x1)^2*(y3*y2+y1^2+y1*sqrt(y3*y2))
*exp((2/3)*arctan((2/3)*y1*sqrt(3)/sqrt(y3*y2)+(1/3)*sqrt(3))*sqrt(3)); }{\[\]}
\end{mapleinput}
\end{maplegroup}
\begin{maplegroup}
\begin{mapleinput}
\mapleinline{active}{1d}{G1 := (1/2)*(y1^2-y2*y3)*(diff(f(x1), x1))/f(x1); }{\[\]}
\end{mapleinput}
\end{maplegroup}
\begin{maplegroup}
\begin{mapleinput}
\mapleinline{active}{1d}{G2 := y2*(2*y1+sqrt(y3*y2))*(diff(f(x1), x1))/(2*f(x1)); }{\[\]}
\end{mapleinput}
\end{maplegroup}
\begin{maplegroup}
\begin{mapleinput}
\mapleinline{active}{1d}{G3 := y3*(2*y1+sqrt(y3*y2))*(diff(f(x1), x1))/(2*f(x1)); }{\[\]}
\end{mapleinput}
\end{maplegroup}
\begin{maplegroup}
\begin{mapleinput}
\mapleinline{active}{1d}{y1 := y[1]; 1; y2 := y[2]; 1; y3 := y[3]}{\[\]}
\end{mapleinput}
\mapleresult
\begin{maplelatex}
\mapleinline{inert}{2d}{y1 := y[1]}{\[\displaystyle {\it y1}\, := \,y_{{1}}\]}
\end{maplelatex}
\mapleresult
\begin{maplelatex}
\mapleinline{inert}{2d}{y2 := y[2]}{\[\displaystyle {\it y2}\, := \,y_{{2}}\]}
\end{maplelatex}
\mapleresult
\begin{maplelatex}
\mapleinline{inert}{2d}{y3 := y[3]}{\[\displaystyle {\it y3}\, := \,y_{{3}}\]}
\end{maplelatex}
\end{maplegroup}
\begin{maplegroup}
\begin{mapleinput}
\mapleinline{active}{1d}{printlevel := 3;                                                          \phantom{ghmmmm} for i to 3 do  for j to i do  for k to j do                                                                                                                                                                                                                \phantom{ghmmmm} Landsberg[i, j, k] := simplify((diff(F, y1))*(diff(G1, y[i], y[j], y[k]))    \phantom{ghmmmm} +(diff(F, y2))*(diff(G2, y[i], y[j], y[k]))+(diff(F, y3))*(diff(G3, y[i], y[j], y[k])))                                                                                  \phantom{ghmmmm}  end do;   end do ;  end do; }{\[\]}
\end{mapleinput}
\mapleresult
\begin{maplelatex}
\mapleinline{inert}{2d}{Landsberg[1, 1, 1] := 0}{\[\displaystyle {\it Landsberg}_{{1,1,1}}\, := \,0\]}
\end{maplelatex}
\mapleresult
\begin{maplelatex}
\mapleinline{inert}{2d}{Landsberg[2, 1, 1] := 0}{\[\displaystyle {\it Landsberg}_{{2,1,1}}\, := \,0\]}
\end{maplelatex}
\mapleresult
\begin{maplelatex}
\mapleinline{inert}{2d}{Landsberg[2, 2, 1] := 0}{\[\displaystyle {\it Landsberg}_{{2,2,1}}\, := \,0\]}
\end{maplelatex}
\mapleresult
\begin{maplelatex}
\mapleinline{inert}{2d}{Landsberg[2, 2, 2] := 0}{\[\displaystyle {\it Landsberg}_{{2,2,2}}\, := \,0\]}
\end{maplelatex}
\mapleresult
\begin{maplelatex}
\mapleinline{inert}{2d}{Landsberg[3, 1, 1] := 0}{\[\displaystyle {\it Landsberg}_{{3,1,1}}\, := \,0\]}
\end{maplelatex}
\mapleresult
\begin{maplelatex}
\mapleinline{inert}{2d}{Landsberg[3, 2, 1] := 0}{\[\displaystyle {\it Landsberg}_{{3,2,1}}\, := \,0\]}
\end{maplelatex}
\mapleresult
\begin{maplelatex}
\mapleinline{inert}{2d}{Landsberg[3, 2, 2] := 0}{\[\displaystyle {\it Landsberg}_{{3,2,2}}\, := \,0\]}
\end{maplelatex}
\mapleresult
\begin{maplelatex}
\mapleinline{inert}{2d}{Landsberg[3, 3, 1] := 0}{\[\displaystyle {\it Landsberg}_{{3,3,1}}\, := \,0\]}
\end{maplelatex}
\mapleresult
\begin{maplelatex}
\mapleinline{inert}{2d}{Landsberg[3, 3, 2] := 0}{\[\displaystyle {\it Landsberg}_{{3,3,2}}\, := \,0\]}
\end{maplelatex}
\mapleresult
\begin{maplelatex}
\mapleinline{inert}{2d}{Landsberg[3, 3, 3] := 0}{\[\displaystyle {\it Landsberg}_{{3,3,3}}\, := \,0\]}
\end{maplelatex}
\end{maplegroup}
\begin{maplegroup}
\begin{mapleinput}
\mapleinline{active}{1d}{Berwald2[2, 2, 2] := simplify(diff(G2, y[2], y[2], y[2]))}{\[ \]}
\end{mapleinput}
\mapleresult
\begin{maplelatex}
\mapleinline{inert}{2d}{}{\[\displaystyle {\it Berwald2}_{{2,2,2}}\, := -3/16\,{\frac {y_{{3}}{\frac {d}{d{\it x1}}}f \left( {\it x1} \right) }{y_{{2}} \sqrt{y_{{3}}y_{{2}}}\\
\mbox{}f \left( {\it x1} \right) }}\]}
\end{maplelatex}
\end{maplegroup}
This shows  that $(M,F)$ is a non-Berwaldian Landsberg manifold.
\end{example}

\medskip

\begin{example}\label{ex.2} Let
$M= \mathbb{R}^3$ and $F$ be given by
$${F}(x,y)=f(x^1)\sqrt{(y^1)^2+(y^2)^2+(y^3)^2+y^1\sqrt{(y^2)^2+(y^3)^2}}\,e^{\frac{1}{\sqrt{3}}\arctan\Big(\frac{2y^1}{\sqrt{3((y^2)^2+(y^3)^2)}}+\frac{1}{\sqrt{3}}\Big)}.$$

The coefficients of the geodesic spray of $F$ are given by
$$G^1=\left((y^1)^2-(y^2)^2-(y^3)^2\right)\frac{f'(x^1)}{2f(x^1)},\quad G^2=\left( y^2(2y^1+\sqrt{(y^2)^2+(y^3)^2})\right)\frac{f'(x^1)}{2f(x^1)},$$
$$ G^3=\left( y^3(2y^1+\sqrt{(y^2)^2+(y^3)^2})\right)\frac{f'(x^1)}{2f(x^1)}.$$
By the same method, as in the above example, we  calculate the Landsberg  and the Berwald tensors as follows
$$L_{ijk}=0, \quad G^h_{ijk}\neq 0, \ \, \text{for example,}\,\, \  G^2_{222}=-\frac{3}{2}\frac{(y^3)^4}{((y^2)^2+(y^3)^2)^{5/2}}\frac{f'(x^1)}{f(x^1)}.$$
Consequently, $(M,F)$ is a non-Berwaldian Landsberg manifold.
\end{example}

\medskip

The following example is $4$-dimensional manifold.
\begin{example}\label{ex.2} Let
$M= \mathbb{R}^4$ and    $F$  be given by
$$F(x,y):=f(x^1)\sqrt{(y^1)^2+y^2y^3+(y^4)^2+y^1\sqrt{y^2y^3+(y^4)^2}}\,e^{\frac{1}{\sqrt{3}} \ \arctan\Big(\frac{2y^1}{\sqrt{3}\sqrt{y^2y^3+(y^4)^2}}+\frac{1}{\sqrt{3}}\Big)}.$$
The coefficients of the geodesic spray of $F$ are given by
$$G^1=\left(y^1)^2-y^2y^3-(y^4)^2\right)\frac{f'(x^1)}{2f(x^1)},\quad G^2=\left(y^2(2y^1+\sqrt{y^2y^3+(y^4)^2})\right)\frac{f'(x^1)}{2f(x^1)},$$
$$ G^3=\left(y^3(2y^1+\sqrt{y^2y^3+(y^4)^2})\right)\frac{f'(x^1)}{2f(x^1)}, \quad  G^4=\left(y^4(2y^1+\sqrt{y^2y^3+(y^4)^2})\right)\frac{f'(x^1)}{2f(x^1)}.$$
By the same way, we have
$$L_{ijk}=0, \quad G^h_{ijk}\neq 0, \, \text{for example,}\,\,  G^2_{222}=-\frac{3}{16}\frac{(y^3)^2(y^2y^3+2(y^4)^2))}{(y^2y^3+(y^4)^2)^{5/2}}\frac{f'(x^1)}{f(x^1)}.$$
Consequently,  $(M,F)$ is a non-Berwaldian Landsberg space.
\end{example}

\subsection{Some  remarks}

 Depending on the examples mentioned above, we have the following remarks:

\medskip

$\bullet$ The spray of a Landsberg space is very specific; in the extreme directions (say, $(\pm 1,0,...,0)$), the  coefficient $G^1$ of the spray  is quadratic  and in the other directions, the  coefficients are not quadratic but they are projectively flat ($G^\mu=Py^\mu$, where $\mu=2,...,n$). This  inspires us two things; one thing is how to obtain new solutions for the unicorn problem and this will be discussed  in the next section in details. The other thing is the following point.

\medskip

$-$ Without loss of generality, we fix the extreme directions of a Landsberg metric as $(\pm 1 , 0,...,0)$, then we have

$$G^1=f_{ij}(x^1)y^iy^j, \quad G^\mu=Py^\mu,$$
where $f_{ij}$ are some functions on $\Real$. Then, the Berwald tensor is given by
$$G^1_{ijk}=0,\quad G^h_{1jk}=0,\quad G^\mu_{\lambda \nu \gamma}=P_{\lambda \nu \gamma} \ y^\mu +P_{\lambda \nu } \ \delta^\mu_\gamma+P_{\nu \gamma  } \ \delta^\mu_\lambda+P_{\gamma\lambda  } \ \delta^\mu_\nu, $$
where $P$ is  defined and  smooth  on an open subset of $\T M$, $P_i:=\dot{\partial}_i P, P_{ij}:=\dot{\partial}_i\dot{\partial}_j P, P_{ijk}:=\dot{\partial}_i\dot{\partial}_j \dot{\partial}_k P$ and moreover, $P_{1j}=0$. Throughout, the Greek letters $\lambda, \mu, \nu, \gamma$ run over  $2,...,n$.

Therefore, the Landsberg tensor takes the form

\begin{equation}
\label{New_Lands_tensor}
L_{\lambda \nu \gamma}=-\frac{1}{2}F(P_{\lambda \nu \gamma} \ \ell_\mu y^\mu +P_{\lambda \nu } \ \ell_\gamma+P_{\nu \gamma  } \ \ell_\lambda+P_{\gamma\lambda  } \ \ell_\nu).
\end{equation}
In this case,  the Landsberg tensor is given in terms of one function $P$ besides the Finsler function. This property of the spray makes the condition, for a metric,  of being Landsberg much weaker; it is possible to find a function $P$ satisfies $L_{\lambda \nu \gamma}=0$ instead of the existence of $n$   functions  $G^i$ such that $L_{ijk}=-\frac{1}{2}F\ell_h G^h_{ijk}=0$.

\bigskip

$\bullet$ The coefficients  of the geodesic spray of the class \eqref{Shen_class} are given by
$$G^i=G_\alpha^i+\frac{c_1k\sqrt{\alpha^2-(\beta/b_0)^2}}{2(1+c_3b_0^2)}\set{ b_0^2y^i-\beta b^i+\frac{c_3k}{c_1}\sqrt{\alpha^2-(\beta/b_0)^2}b^i } .$$
Choosing $\alpha=f(x^1)\sqrt{(y^1)^2+\varphi(y^\mu)}$, $\beta=f(x^1)y^1$ and by Subsection 4.1, the coefficients of the geodesic spray are given by
\begin{equation}
\label{Lands_type_Spray}
G^1=\left( (y^1)^2-\varphi(y^\mu)+\frac{c_3k(\alpha^2-\beta^2)}{(1+c_3)}\right)\frac{f'(x^1)}{2f(x^1)},\,\, G^\mu=\frac{f'(x^1)}{f(x^1)}\left( y^1+ \frac{c_1k\sqrt{\alpha^2-\beta^2}}{2(1+c_3)}\right)y^\mu
\end{equation}

\section{New  solutions for the unicorn problem}

 As we discussed in the previous section, by the help of computer and  Maple program, we have  examples in which  the spray coefficients are very simple. Deforming  these sprays, that we got, in a very specific way and  by using the metrizability of the new sprays we obtain the solutions. The interesting thing in this section is that we  establish  solutions for the unicorn's problem that have  never  been investigated.

According to our previous discussion, the spray of a Landsberg space (non-regular), without loss of generality,  has a very specific form. That is, in the direction of singularity the spray is quadratic and in the other directions it is  projectively flat. This inspires us to get new solutions for the unicorn problem. The idea, as will be discussed soon, is to start by a one spray of a Landsberg metric and then make a small deformation in it. Checking the metrizability of the new spray, if it is metrizable then we get a solution. Sure not any deformations will yield a new Landsberg metric but we have to try different deformations.

\subsection{Special Riemannian metric}
From now on, without loss of generality, we use   $\beta=f(x^1)y^1$ and $\alpha=f(x^1)\sqrt{(y^1)^2+\varphi(\hat{y})}$, where $f(x^1)$ is a positive function on $\Real$ and $\varphi$ is arbitrary quadratic function in $\hat{y}$ and $\hat{y}$ stands for the variables $y^2,...,y^n$.  The function $\varphi$ should be chosen in  such a way  the metric tensor of $\alpha$ is non-degenerate. Precisely, $\varphi$ can be written in the form $\varphi=c_{\lambda\mu}y^\lambda y^\mu$, where $c_{\lambda\mu}$ is a non-singular symmetric $n-1\times n-1$ matrix of arbitrary constants.

 In this case, the components of the metric tensor $a_{ij}$ (and its inverse $a^{ij}$, resp.) of $\alpha$  are given by
$$a_{11}=f(x^1)^2, \quad a_{1\mu}=0, \quad a_{\lambda\mu}=f(x^1)^2c_{\lambda\mu}$$
$$a^{11}=\frac{1}{f(x^1)^2}, \quad a^{1\mu}=0, \quad a^{\lambda\mu}=\frac{1}{f(x^1)^2}c^{\lambda\mu},$$
where $c^{\lambda\mu}$ is the inverse matrix of $c_{\lambda\mu}$.
Using \eqref{geodesic_spray}, we can find the geodesic spray $G_\alpha^i$ of $\alpha$, as follows
$$G_\alpha^1=\left(\frac{2f(x^1)^2(y^1)^2-\alpha^2}{2f(x^1)^3}\right)f'(x^1), \quad G_\alpha^\mu=\frac{f'(x^1)}{f(x^1)}y^1y^\mu.$$
The coefficients $\gamma^h_{ij}:=\dot{\partial}_i\dot{\partial}_j G_\alpha^h$ of Levi-Civita connection of $\alpha$ are given by
$$\gamma^1_{11}=\frac{f'(x^1)}{f(x^1)},\quad \gamma^1_{\lambda\mu}=-\frac{f'(x^1)}{f(x^1)}c_{\lambda\mu},\quad \gamma^\mu_{1\lambda}=\frac{f'(x^1)}{f(x^1)}\delta^\mu_\lambda,\quad \gamma^1_{1\mu}= \gamma^\mu_{11}= \gamma^\mu_{\lambda\nu}=0.$$
Since, $b_1=f(x^1)$ and $b_\mu=0$,  one can, easily, have
$$b^1=\frac{1}{f(x^1)}, \quad b^\mu=0,\quad b^2=b_ib^i=1,$$
$$b_{1|1}=0,\quad b_{1|\mu}=0,\quad b_{\mu|1}=0,\quad b_{\lambda|\mu}=f'(x^1) c_{\lambda\mu} $$
from which we get $b_{i|j}=b_{j|i}$ and hence
$$s_{ij}=0,\quad r_{ij}=b_{i|j},\quad r_{00}=\frac{(\alpha^2-\beta^2)f'(x^1)}{f(x^1)^2}.$$
Moreover,
$$b_{i|j}=k(a_{ij}-b_ib_j), \quad k=\frac{f'(x^1)}{f(x^1)^2}.$$

\subsection{First class}

In this subsection, we give a new and very simple class of non-Berwaldain Landsberg metrics. We have a remark about  this class at the end of this paper. We use Maple program and metrizability criteria to get the solutions that we are looking for. From now on, the following procedure will be pursued throughout.   Let's start with the following two examples.

\begin{example}\label{ex.4} Let
$M= \mathbb{R}^3$. Consider the coefficients $G^i$ of a  spray $S$ are given by
$$G^1=\left((y^1)^2-y^2y^3\right)\frac{f'(x^1)}{2f(x^1)},\quad G^2=Py^2,\quad G^3=Py^3,$$
where the function $P$ is given by
$$ P=\left(y^1+\sqrt{y^2y^3}\right)\frac{f'(x^1)}{f(x^1)}.$$

It is known that a spray $S$ is Finsler metrizable, by a Finsler function $F$, if and only if  the system \eqref{metrizable_system} has non-degenerate solutions.
 We add the Landsberg condition  $\ell_h G^h_{ijk}=0$, $\ell_h=\dot{\partial}_hF$ to above system and then solve it by Maple program as follows:

\begin{maplegroup}
\begin{mapleinput}
\mapleinline{active}{1d}{restart}{\[\]}
\end{mapleinput}
\end{maplegroup}
\begin{maplegroup}
\begin{mapleinput}
\mapleinline{active}{1d}{F := f(x1)*f1(y1, y2, y3); }{\[\]}
\end{mapleinput}
\end{maplegroup}
\begin{Maple Normal}{
\begin{Maple Normal}{
}\end{Maple Normal}
}\end{Maple Normal}
\begin{maplegroup}
\begin{mapleinput}
\mapleinline{active}{1d}{P := (y1+sqrt(y2*y3))*(diff(f(x1), x1))/f(x1); }{\[\]}
\end{mapleinput}
\end{maplegroup}
\begin{maplegroup}
\begin{mapleinput}
\mapleinline{active}{1d}{G1 := (1/2)*(y1^2-y2*y3)*(diff(f(x1), x1))/f(x1); }{\[\]}
\end{mapleinput}
\end{maplegroup}
\begin{maplegroup}
\begin{mapleinput}
\mapleinline{active}{1d}{G2 := P*y2; }{\[\]}
\end{mapleinput}
\end{maplegroup}
\begin{maplegroup}
\begin{mapleinput}
\mapleinline{active}{1d}{G3 := P*y3; }{\[\]}
\end{mapleinput}
\end{maplegroup}
\begin{maplegroup}
\begin{mapleinput}
\mapleinline{active}{1d}{}{\[\]}
\end{mapleinput}
\end{maplegroup}
\begin{maplegroup}
\begin{mapleinput}
\mapleinline{active}{1d}{sys1 := [                                                                   \phantom{mmmmmmmm}    diff(F, x1)-(diff(G1, y1))*(diff(F, y1))-(diff(G2, y1))*(diff(F, y2))-(diff(G3, y1))*(diff(F, y3)) = 0,                                                                 \phantom{mmmmmmmm} diff(F, x2)-(diff(G1, y2))*(diff(F, y1))-(diff(G2, y2))*(diff(F, y2))-(diff(G3, y2))*(diff(F, y3)) = 0,                                                     \phantom{mmmmmmmm}   diff(F, x3)-(diff(G1, y3))*(diff(F, y1))-(diff(G2, y3))*(diff(F, y2))-(diff(G3, y3))*(diff(F, y3)) = 0,                                                                  \phantom{mmmmmmmm}  (diff(F, y2))*(diff(G2, y2, y2, y2))+(diff(F, y3))*(diff(G3, y2, y2, y2)) = 0,                                                                                     \phantom{mmmmmmmm}         (diff(F, y2))*(diff(G2, y2, y2, y3))+(diff(F, y3))*(diff(G3, y2, y2, y3)) = 0,                                                                               \phantom{mmmmmmmm}   (diff(F, y2))*(diff(G2, y3, y3, y2))+(diff(F, y3))*(diff(G3, y3, y3, y2)) = 0,                                                                                      \phantom{mmmmmmmm}           (diff(F, y2))*(diff(G2, y3, y3, y3))+(diff(F, y3))*(diff(G3, y3, y3, y3)) = 0,                                                                         \phantom{mmmmmmmm}  y1*(diff(F, y1))+y2*(diff(F, y2))+y3*(diff(F, y3)) = F]; }{\[\]}
\end{mapleinput}
\end{maplegroup}
\begin{maplegroup}
\begin{mapleinput}
\mapleinline{active}{1d}{pdsolve(sys1)}{\[{\it pdsolve} \left( {\it sys1} \right) \]}
\end{mapleinput}
\mapleresult
\begin{maplelatex}
\mapleinline{inert}{2d}{{f(x1) = 0, f1(y1, y2, y3) = f1(y1, y2, y3)}, {f(x1) = _C1, f1(y1, y2, y3) = _F1(y2/y1, y3/y1)*y1}, {f(x1) = f(x1), f1(y1, y2, y3) = _C1*exp(y1/(y1+sqrt(y2*y3)))*(y1+sqrt(y2*y3))}}{\[\displaystyle   \left\{ f \left( {\it x1} \right) =f \left( {\it x1} \right) ,{\it f1} \left( {\it y1},{\it y2},{\it y3}\\
\mbox{} \right) ={\it \_C1}\,{{\rm e}^{{\frac {{\it y1}}{{\it y1}+ \sqrt{{\it y2}\,{\it y3}\\
\mbox{}}}}}} \left( {\it y1}+ \sqrt{{\it y2}\,{\it y3}\\
\mbox{}} \right)  \right\} \]}
\end{maplelatex}
\end{maplegroup}

Consequently,  the spray $S$ is metrizable by the non-regular Finsler function
$${F}(x,y)=f(x^1)(y^1+\sqrt{y^2y^3})\,e^{\frac{y^1}{y^1+\sqrt{y^2y^3}}}.$$
Moreover, one  can see easily
$$G^h_{ijk}\neq 0\,\quad \left(G^2_{222}=-\frac{3}{8}\frac{y^3}{y^2\sqrt{y^2y^3}}\frac{f'(x^1)}{f(x^1)}\right),\quad L_{ijk}=0.$$
Hence, $(M,F)$ is a non Berwaldain Landsberg space.
\end{example}

\medskip

Considering a   more general case, we have the following.

\begin{example}\label{New_example_1a} Let
$M= \mathbb{R}^3$ and the spray  $S$ be given by
$$G^1=\frac{a^2(y^1)^2-y^2y^3}{2a^2}\frac{f'(x^1)}{f(x^1)},\quad G^2=Py^2,\quad G^3=Py^3,$$
where the function $P$ is given by
$$ P=\left(y^1+ \frac{1}{a}\sqrt{y^2y^3}\right)\frac{f'(x^1)}{f(x^1)}.$$
By solving the system for metrizability and Landsberg restriction, we have  the following non regular Finsler function
$${F}(x,y)=f(x^1)(a y^1+\sqrt{y^2y^3})\,e^{\frac{ay^1}{ay^1+\sqrt{y^2y^3}}}.$$
Moreover, one  can see easily that
$$G^h_{ijk}\neq 0\,\quad \left(G^2_{222}=-\frac{3}{8a}\frac{y^3}{y^2\sqrt{y^2y^3}}\frac{f'(x^1)}{f(x^1)}\right),\quad L_{ijk}=0.$$
Hence, $(M,F)$ is a non Berwaldain Landsberg space.
\end{example}

\medskip

\begin{example}\label{New_example_1b} Let
$M= \mathbb{R}^3$ and the spray  $S$ be given by
$$G^1=\frac{a^2(y^1)^2-(y^2)^2-(y^3)^2}{2a^2}\frac{f'(x^1)}{f(x^1)},\quad G^2=Py^2,\quad G^3=Py^3,$$
where the function $P$ is given by
$$ P=\left( y^1+ \frac{1}{a}\sqrt{(y^2)^2+(y^3)^2}\right)\frac{f'(x^1)}{f(x^1)}.$$
The spray $S$ is metrizable by the non-regular Finsler function
$${F}(x,y)=f(x^1)(a y^1+\sqrt{(y^2)^2+(y^3)^2})\,e^{\frac{ay^1}{ay^1+\sqrt{(y^2)^2+(y^3)^2}}}.$$
Moreover,
$$G^h_{ijk}\neq 0\,\quad \left(G^2_{222}=\frac{3}{a}\frac{(y^3)^4}{((y^2)^2+(y^3)^2)^{5/2}}\frac{f'(x^1)}{f(x^1)}\right),\quad L_{ijk}=0.$$
Hence, $(M,F)$ is a non-Berwaldain Landsberg space.
\end{example}

\medskip

Summing up, we introduce the first class of Landsberg metrics which are not Berwaldian.

\begin{theorem}\label{First_class}
For $n\geq 3$, the  class
$${F}=\left(a \beta+\sqrt{\alpha^2-\beta^2}\right)\,e^{\frac{a \beta}{a\beta+\sqrt{\alpha^2-\beta^2}}}$$
on an $n$-dimensional manifold $M$ is a class of Landsberg metrics which are not Berwaldian, where  $\beta$ and $\alpha$ are given in Subsection 4.1 and $a\neq 0$. Moreover, the geodesic spray of $F$ is given by
$$G^1=\left(\frac{2f(x^1)^2(y^1)^2-\alpha^2}{2f(x^1)^2}+\frac{a^2-1}{2a^2}\frac{\alpha^2-\beta^2}{f(x^1)^2}\right)\frac{f'(x^1)}{f(x^1)},\quad G^\mu=Py^\mu,$$
where the function $P$ is given by
$$ P=\left(y^1+ \frac{1}{af(x^1)}\sqrt{\alpha^2-\beta^2}\right)\frac{f'(x^1)}{f(x^1)}.$$
\end{theorem}

\begin{proof}
To find the geodesic spray of $F$, we have to calculate the functions $Q$ and $\Theta$ which are defined by \eqref{Q} and \eqref{Theta}, resp. One can calculate them by hand, or instead by using the Maple program.  We have the following
$$\phi(s)=\left(a s+\sqrt{1-s^2}\right)\,e^{\frac{a s}{as+\sqrt{1-s^2}}},$$
$$Q=2a\sqrt{1-s^2}+(a^2-1)s,\quad \Theta=\frac{1}{a\sqrt{1-s^2}},$$
moreover, we have
$$\frac{Q'}{Q-sQ'}=\frac{(a^2-1)\sqrt{1-s^2}-2as}{2a}.$$
Finally, making use of  Subsection 4.1 and \eqref{spray_alpha-beta_metric}, we get the formulae for $G^1$ and $G^\mu$.

Since the geodesic spay of $F$ is quadratic in one direction and projectively flat in the others,  the Landsberg tensor takes the form \eqref{New_Lands_tensor}. Now, to compute the Landsberg tensor, we have the following
\begin{equation}
\label{L_mu_1st_class}
\ell_\mu=\frac{f(x^1)\, e^{\frac{a \beta}{a\beta+\sqrt{\alpha^2-\beta^2}}}}{2(a\beta+\sqrt{\alpha^2-\beta^2})}\,\alpha^2_{y^\mu},
\end{equation}
where, for simplicity, we use the notation $\alpha^2_{y^\mu}:=\dot{\partial}_\mu \alpha^2$. Also, we get
$$P_\mu=\frac{1}{2a}\frac{\alpha^2_{y^\mu}}{\sqrt{\alpha^2-\beta^2}}\frac{f'(x^1)}{f(x^1)^2},$$
\begin{equation}
\label{P_lambda_mu_1st_class}
P_{\lambda\mu}=\frac{1}{2a}\left(\frac{\alpha^2_{y^\lambda y^\mu}}{\sqrt{\alpha^2-\beta^2}}-\frac{\alpha^2_{y^\lambda }\alpha^2_{ y^\mu}}{2(\alpha^2-\beta^2)^{3/2}}\right)\frac{f'(x^1)}{f(x^1)^2},
\end{equation}

\begin{equation}
\label{P_sigma_lambda_mu_1st_class}
P_{\lambda\mu\nu}=\frac{1}{2a}\left(\frac{3\alpha^2_{y^\lambda }\alpha^2_{ y^\mu}\alpha^2_{ y^\nu}}{4(\alpha^2-\beta^2)^{5/2}}-\frac{\alpha^2_{ y^\lambda y^\mu}\alpha^2_{ y^\nu}}{2(\alpha^2-\beta^2)^{3/2}}-\frac{\alpha^2_{ y^\lambda y^\nu}\alpha^2_{ y^\mu}}{2(\alpha^2-\beta^2)^{3/2}}-\frac{\alpha^2_{ y^\nu y^\mu}\alpha^2_{ y^\lambda}}{2(\alpha^2-\beta^2)^{3/2}}\right)\frac{f'(x^1)}{f(x^1)^2}.
\end{equation}
Using the fact that $y^\mu\alpha_{y^\mu}=\alpha-y^1\alpha_{y^1}=\frac{\alpha^2-\beta^2}{\alpha}$ together with  \eqref{L_mu_1st_class}, we have
\begin{equation}
\label{y_mu_L_mu_1st}
y^\mu \ell_\mu=f(x^1)\frac{\alpha^2-\beta^2}{a\beta+\sqrt{\alpha^2-\beta^2}}\,e^{\frac{a \beta}{a\beta+\sqrt{\alpha^2-\beta^2}}}.
\end{equation}
By substituting from \eqref{L_mu_1st_class}, \eqref{P_lambda_mu_1st_class}, \eqref{P_sigma_lambda_mu_1st_class}, \eqref{y_mu_L_mu_1st} into \eqref{Lands_type_Spray}, we find that the Landsberg tensor vanishes and hence the metric is Landsbergian. Moreover, the spray is quadratic only if the function $P$ is linear. But in this case, the function $\sqrt{\alpha^2-\beta^2}$ is linear and therefore  the Riemannian metric tensor associated with  $\alpha$ will be degenerate; indeed,
$$\sqrt{\alpha^2-\beta^2}=\sqrt{\phi}=a_\mu y^\mu,$$
where $a_\mu$ some constants.
Now, by differentiation the above equality with respect to $y^\nu$ and $y^\lambda$, we have
$$c_{\lambda\nu}=\frac{1}{\phi}c_{\mu\nu}y^\mu c_{\lambda\delta}y^\delta=\theta_\lambda\theta_\nu,$$
where $\theta_\nu=\frac{1}{\sqrt{\phi}}c_{\mu\nu}y^\mu$. But this kind of matrices has rank $1$ and since $n>2$, then the determinant of the matrix $c_{\mu\nu}$ vanishes. Hence, the metric tensor associated with $\alpha$ is degenerate. Consequently,  $P$ can not be linear and therefore the spray will not be quadratic which means that the metric is not Berwald.
\end{proof}

\subsection{Second class}

In this subsection, we give another class of non-Berwaldian Landsberg metrics. We will discuss it at the end of this paper.

\begin{example}\label{New_example_2a} Let
$M= \mathbb{R}^3$ and the spray  $S$ be given by
$$G^1=\frac{(a^2-1)(y^1)^2-y^2y^3}{2(a^2-1)}\frac{f'(x^1)}{f(x^1)},\quad G^2=Py^2,\quad G^3=Py^3,$$
where the function $P$ is given by
$$ P= \left(y^1+ \frac{a}{a^2-1}\sqrt{y^2y^3}\right)\frac{f'(x^1)}{f(x^1)}.$$
By the same manner, as we did in the previous subsection, we obtain that   the spray $S$ is metrizable by the non-regular Finsler function
$${F}(x,y)=f(x^1)\left((a+1) y^1+\sqrt{y^2y^3}\right)^{(1+a)/2}\,\left((a-1) y^1+\sqrt{y^2y^3}\right)^{(1-a)/2}.$$
Moreover, one  can  easily obtain
$$G^h_{ijk}\neq 0\,\quad \left(G^2_{222}=-\frac{3a}{8(a^2-1)}\frac{y^3}{y^2\sqrt{y^2y^3}}\frac{f'(x^1)}{f(x^1)}\right),\quad L_{ijk}=0.$$
Hence, $(M,F)$ is a non-Berwaldain Landsberg space.
\end{example}

\medskip

\begin{example}\label{New_example_2b} Let
$M= \mathbb{R}^3$ and the spray  $S$ be given by
$$G^1=\frac{(a^2-1)(y^1)^2-(y^2)^2-(y^3)^2}{2(a^2-1)}\frac{f'(x^1)}{f(x^1)},\quad G^2=Py^2,\quad G^3=Py^3,$$
where the function $P$ is given by
$$ P= \left(y^1+ \frac{a}{a^2-1}\sqrt{(y^2)^2+(y^3)^2}\right)\frac{f'(x^1)}{f(x^1)}.$$
By the same  procedure,  we conclude that    the spray $S$ is metrizable by the non-regular Finsler function
$${F}(x,y)=f(x^1)\left((a+1) y^1+\sqrt{(y^2)^2+(y^3)^2}\right)^{(1+a)/2}\,\left((a-1) y^1+\sqrt{(y^2)^2+(y^3)^2}\right)^{(1-a)/2}.$$
Moreover, one  can obtain easily
$$G^h_{ijk}\neq 0\,\quad \left(G^2_{222}=-\frac{3a}{(a^2-1)}\frac{(y^3)^4}{((y^2)^2+(y^3)^2)^{(5/2)}}\frac{f'(x^1)}{f(x^1)}\right),\quad L_{ijk}=0.$$
Hence, $(M,F)$ is a non-Berwaldain Landsberg space.
\end{example}

\medskip

Generally, we have the following class.

\begin{theorem}
For $n\geq 3$, the class
$${F}=\left((a+1)\beta +\sqrt{\alpha^2-\beta^2}\right)^{(1+a)/2}\,\left((a-1) \beta+\sqrt{\alpha^2-\beta^2}\right)^{(1-a)/2}$$
 on an $n$-dimensional manifold $M$ is a class of Landsberg metrics which are not Berwaldian,  where  $\beta$ and $\alpha$ are given in Subsection 4.1 and $a\neq 0 ,\pm 1$. The geodesic spray of $F$ is given by
$$G^1=\left(\frac{2f(x^1)^2(y^1)^2-\alpha^2}{2f(x^1)^2}+\frac{a^2-2}{2(a^2-1)}\frac{\alpha^2-\beta^2}{f(x^1)^2}\right)\frac{f'(x^1)}{f(x^1)},\quad G^\mu=Py^\mu,$$
where the function $P$ is given by
$$ P= \left(y^1+ \frac{a}{a^2-1}\sqrt{\alpha^2-\beta^2}\right)\frac{f'(x^1)}{f(x^1)}.$$
Moreover, when $a=0$ the metric is Riemannian and when $a=\pm 1$ the metric is singular ($det(g)=0$).
\end{theorem}

\begin{proof}
As we did in the proof of the Theorem \ref{First_class}, the functions $Q$ and $\Theta$ can be calculated by using the fact that
$$\phi(s)=\left((a+1)s +\sqrt{1-s^2}\right)^{(1+a)/2}\,\left((a-1) s+\sqrt{1-s^2}\right)^{(1-a)/2}.$$
We have
$$Q=2a\sqrt{1-s^2}+(a^2-2)s,\quad \Theta=\frac{a}{(a^2-1)\sqrt{1-s^2}},$$
moreover, we get
$$\frac{Q'}{Q-sQ'}=\frac{(a^2-2)\sqrt{1-s^2}-2as}{2a}.$$
Finally, making use of  Subsection 4.1 and \eqref{spray_alpha-beta_metric}, we have the formulae for $G^1$ and $G^\mu$.
The rest of the proof can be followed in a similar manner as in the proof of Theorem \ref{First_class}.
\end{proof}

\subsection{Third class}

In this subsection, we establish another class of non-Berwaldian Landsberg metrics. Also, we will discuss it at the end of this paper.
\begin{example}\label{New_example_4} Let
$M= \mathbb{R}^3$ the spray  $S$ be given by
$$G^1=\frac{a^2(y^1)^2-2y^2y^3}{2a^2}\frac{f'(x^1)}{f(x^1)},\quad G^2=Py^2,\quad G^3=Py^3,$$
where the function $P$ is given by
$$ P= \left(y^1+ \frac{3}{2a}\sqrt{y^2y^3}\right)\frac{f'(x^1)}{f(x^1)}.$$
By the same manner, as we did in Example 4.1, we obtain that   the spray $S$ is metrizable by the non-regular Finsler function
$${F}(x,y)=f(x^1)\left(a y^1+\frac{y^2y^3}{ay^1+2\sqrt{y^2y^3}}\right).$$
Moreover, one  can show easily
$$G^h_{ijk}\neq 0\,\quad \left(G^2_{222}=-\frac{9}{16a}\frac{y^3}{y^2\sqrt{y^2y^3}}\frac{f'(x^1)}{f(x^1)}\right),\quad L_{ijk}=0.$$
Hence, $(M,F)$ is a non-Berwaldain Landsberg space.
\end{example}

\medskip

\begin{example}\label{New_example_5} Let
$M= \mathbb{R}^3$.
Let a spray  $S$ be given by
$$G^1=\frac{a^2(y^1)^2-2(y^2)^2-2(y^3)^2}{2a^2}\frac{f'(x^1)}{f(x^1)},\quad G^2=Py^2,\quad G^3=Py^3,$$
where the function $P(y^2,y^3)$ is given by
$$ P= \left(y^1+ \frac{3}{2a}\sqrt{(y^2)^2+(y^3)^2}\right)\frac{f'(x^1)}{f(x^1)}.$$
By solving the system for metrizability and Landsberg properties, we get that
  the spray $S$ is metrizable by the non regular Finsler function
$${F}(x,y)=f(x^1)\left(a y^1+\frac{(y^2)^2+(y^3)^2}{ay^1+2\sqrt{(y^2)^2+(y^3)^2}}\right).$$
Moreover, one  can see easily
$$G^h_{ijk}\neq 0\,\quad \left(G^2_{222}=\frac{9}{2a}\frac{(y^3)^4}{((y^2)^2+(y^3)^2)^{5/2}}\frac{f'(x^1)}{f(x^1)}\right),\quad L_{ijk}=0.$$
Hence, $(M,F)$ is a non Berwaldain Landsberg space.
\end{example}

\medskip
Generally, we have the following class.
\begin{theorem}
For $n\geq 3$, the class
$${F}=f(x^1)\left(a\beta+\frac{\alpha^2-\beta^2}{a\beta+2\sqrt{\alpha^2-\beta^2}}\right)$$
on an $n$-dimensional manifold $M$ is a class of Landsberg metrics which are not Berwaldian,  where  $\beta$ and $\alpha$ are given in Subsection 4.1 and $a\neq 0$. The geodesic spray of $F$ is given by
$$G^1=\left(\frac{2f(x^1)^2(y^1)^2-\alpha^2}{2f(x^1)^2}+\frac{a^2-2}{2a^2}\frac{\alpha^2-\beta^2}{f(x^1)^2}\right)\frac{f'(x^1)}{f(x^1)},\quad G^\mu=Py^\mu,$$
where the function $P$ is given by
$$ P= \left(y^1+ \frac{3}{2a}\sqrt{\alpha^2-\beta^2}\right)\frac{f'(x^1)}{f(x^1)}.$$
Moreover, when $a=0$ the metric is singular ($det(g)=0$).
\end{theorem}
\begin{proof}
As we did in the proof of the Theorem \ref{First_class}, the functions $Q$ and $\Theta$ can be calculated by using the fact that
$$\phi(s)=\left(as+\frac{1-s^2}{as+2\sqrt{1-s^2}}\right).$$
We get the following
$$Q=\frac{3a}{2}\sqrt{1-s^2}+\frac{a^2-2}{2}s, \quad \Theta=\frac{3}{2a\sqrt{1-s^2}}, $$
moreover, we have
$$ \frac{Q'}{Q-sQ'}=-s+\frac{(a^2-1)\sqrt{1-s^2}}{3a}.$$
Finally, making use of  Subsection 4.1 and \eqref{spray_alpha-beta_metric}, we have the formulae for $G^1$ and $G^\mu$. The rest of the proof can be followed in a similar manner as in the proof of Theorem \ref{First_class}.
\end{proof}

\subsection{Fourth class: the most general class}

Based on the Maple program and NF-package, we give a very simple formula for the most general class of Landsberg non-Berwaldian metrics \eqref{Shen_class} which is obtained by Shen \cite{Shen_example}.

\begin{theorem} \label{4th_class}
 Let   $S$ be a spray on an $n$-dimensional manifold $M$ ($n\geq 3$)and given by
$$G^1=\left(\frac{2f(x^1)^2(y^1)^2-\alpha^2}{2f(x^1)^2}+\frac{q}{2(1+q)}\frac{\alpha^2-\beta^2}{f(x^1)^2}\right)\frac{f'(x^1)}{f(x^1)},\quad G^\mu=Py^\mu,$$
where the function $P$ is given by
$$ P=\left(y^1+\frac{p}{2(1+q)}\frac{\sqrt{\alpha^2-\beta^2}}{f(x^1)}\right)\frac{f'(x^1)}{f(x^1)}$$
is Finsler metrizable by the Finsler metric
$$F=\sqrt{\alpha^2+p\beta\sqrt{\alpha^2-\beta^2}+q\beta^2}\,e^{\frac{p}{\sqrt{p^2-4q-4}}\, \text{arctanh}\left(\frac{p\beta+2\sqrt{\alpha^2-\beta^2}}{\beta\sqrt{p^2-4q-4}}\right)}.$$
This is a class of Landsberg metrics which are not Berwaldian,  where  $\beta$ and $\alpha$ are given in Subsection 4.1 and $p\neq 0$. Moreover, when $p^2-4q-4=0$, the class reduced to the first class.
\end{theorem}

\begin{proof}
To find the geodesic spray of $F$, we have
$$\phi(s)=\sqrt{1+ps\sqrt{1-s^2}+qs^2}\,e^{\frac{p}{\sqrt{p^2-4q-4}}\, \text{arctanh}\left(\frac{ps+2\sqrt{1-s^2}}{s\sqrt{p^2-4q-4}}\right)}.$$
The functions $Q$ and $\Theta$ which defined by \eqref{Q} and \eqref{Theta}, are given by
$$Q=p\sqrt{1-s^2}+qs, \quad \Theta=\frac{p}{2(1+q)\sqrt{1-s^2}}.$$
Moreover, we get that
$$  \frac{Q'}{Q-sQ'}=-s+\frac{q\sqrt{1-s^2}}{p}.$$

Finally, making use of  Subsection 4.1 and \eqref{spray_alpha-beta_metric}, we have the formulae for $G^1$ and $G^\mu$.

The rest of the proof can be followed in a similar manner as we did in the proof of Theorem \eqref{First_class}. However, we would like to make it easier in the following way.

Differentiating the Finser function $F$ with respect to $y^\mu$, we have
\begin{equation}
\label{L_mu_4st_class}
\ell_\mu=\Phi\,\alpha^2_{y^\mu},
\end{equation}
where,  $\alpha^2_{y^\mu}=\dot{\partial}_\mu \alpha^2$ and $\Phi$ is given by
$$\Phi:=\frac{1}{2}\frac{1}{\sqrt{\alpha^2+p\beta\sqrt{\alpha^2-\beta^2}+q\beta^2}}\,e^{\frac{p}{\sqrt{p^2-4q-4}}\, \text{arctanh}\left(\frac{p\beta+2\sqrt{\alpha^2-\beta^2}}{\beta\sqrt{p^2-4q-4}}\right)}. $$

Using the fact that $y^\mu\alpha_{y^\mu}=\alpha-y^1\alpha_{y^1}=\frac{\alpha^2-\beta^2}{\alpha}$ together with  \eqref{L_mu_4st_class}, we have
\begin{equation}
\label{y_mu_L_mu_4st}
y^\mu \ell_\mu=2\Phi(\alpha^2-\beta^2).
\end{equation}

 Also, we get
$$P_\mu=\frac{p}{4(1+q)}\frac{\alpha^2_{y^\mu}}{\sqrt{\alpha^2-\beta^2}}\frac{f'(x^1)}{f(x^1)^2},$$
\begin{equation}
\label{P_lambda_mu_4st_class}
P_{\lambda\mu}=\frac{p}{4(1+q)}\left(\frac{\alpha^2_{y^\lambda y^\mu}}{\sqrt{\alpha^2-\beta^2}}-\frac{\alpha^2_{y^\lambda }\alpha^2_{ y^\mu}}{2(\alpha^2-\beta^2)^{3/2}}\right)\frac{f'(x^1)}{f(x^1)^2},
\end{equation}

\begin{equation}
\label{P_sigma_lambda_mu_4st_class}
P_{\lambda\mu\nu}=\frac{p}{4(1+q)}\left(\frac{3\alpha^2_{y^\lambda }\alpha^2_{ y^\mu}\alpha^2_{ y^\nu}}{4(\alpha^2-\beta^2)^{5/2}}-\frac{\alpha^2_{ y^\lambda y^\mu}\alpha^2_{ y^\nu}}{2(\alpha^2-\beta^2)^{3/2}}-\frac{\alpha^2_{ y^\lambda y^\nu}\alpha^2_{ y^\mu}}{2(\alpha^2-\beta^2)^{3/2}}-\frac{\alpha^2_{ y^\nu y^\mu}\alpha^2_{ y^\lambda}}{2(\alpha^2-\beta^2)^{3/2}}\right)\frac{f'(x^1)}{f(x^1)^2}.
\end{equation}

By substituting from \eqref{L_mu_4st_class}, \eqref{P_lambda_mu_4st_class}, \eqref{P_sigma_lambda_mu_4st_class}, \eqref{y_mu_L_mu_4st} into \eqref{Lands_type_Spray}, we find that the Landsberg tensor vanishes and hence the metric is Landsbergian. Moreover, the spray is quadratic if only if the function $P$ is linear. As we discussed at the end of the proof of Theorem \eqref{First_class}, $P$ can not be linear and hence the manifold $(M,F)$ is not Berwald.
\end{proof}

\begin{remark} It should be noted that  the class given in Theorem \ref{4th_class} is equivlaent to the class obtained by Z. Shen \cite{Shen_example}. Indeed,  Shen's class is given by $F=\alpha \phi(s)$, where
$$\phi(s)= \,exp\left(\int \frac{c_1\sqrt{b^2-s^2}+c_2 s}{s(c_1\sqrt{b^2-s^2}+c_2 s)+1}ds\right). $$
For the sake of simplicity, first, let's consider the case when $b^2=1$, then we have
\begin{eqnarray*}
\frac{c_1\sqrt{1-s^2}+c_2 s}{s(c_1\sqrt{1-s^2}+c_2 s)+1}&=&\frac{1}{2}\frac{2c_2s\sqrt{1-s^2}-2c_1 s^2+c_1}{\sqrt{1-s^2}(c_1s\sqrt{1-s^2}+c_2 s^2+1)}\\&&
+\frac{1}{2}\frac{c_1}{\sqrt{1-s^2}(c_1s\sqrt{1-s^2}+c_2 s^2+1)}.
\end{eqnarray*}
In the above equality, one can see that the first fraction in the right hand side, the numerator is the derivative of the denumerator. Moreover, the second fraction can be rewritten in the form
$$\frac{1}{\sqrt{1-s^2}(c_1s\sqrt{1-s^2}+c_2 s^2+1)+1}=\frac{1}{r_o^2-u^2}du,$$
$$r_o:=\frac{c_1}{\sqrt{c_1^2-4c_2-4}}, \quad u:=\frac{c_1s+2\sqrt{1-s^2}}{s\sqrt{c_1^2-4c_2-4}}.$$
Consequently, we $\phi(s)$ is given by
$$\phi(s)=\sqrt{1+c_1s\sqrt{1-s^2}+c_2s^2}\,e^{\frac{c_1}{\sqrt{c_1^2-4c_2-4}}\, \text{arctanh}\left(\frac{c_1s+2\sqrt{1-s^2}}{s\sqrt{c_1^2-4c_2-4}}\right)}. $$
For the general case, one can rewrite $\phi(s)$ in the form
$$\phi=exp\left(\int \frac{c_1'\sqrt{1-s'^2}+c_2' s'}{s'(c_1'\sqrt{1-s'^2}+c_2' s)+1}ds'\right),$$
where $c_1':=b_oc_1$, $c_2':=b_oc_2$, $s':=\frac{s}{b_o}$. In a similar manner,   $\phi$ is given by
$$\phi=\sqrt{1+c_1's'\sqrt{1-s'^2}+c_2's'^2}\,e^{\frac{c_1'}{\sqrt{c_1'^2-4c_2'-4}}\, \text{arctanh}\left(\frac{c_1's'+2\sqrt{1-s'^2}}{s'\sqrt{c_1'^2-4c_2'-4}}\right)}.$$
\end{remark}

\section{Concluding remarks}

We end this work by the following  outstanding  remarks.

$\bullet$ The first class is a new and very simple class of non Berwaldian Landsberg spaces. Furthermore, this class  is not included in  Shen's solutions. Comparing this class with Shen's one, we find that the constants $c_1=2a$ and $c_3=a^2-1$. At these values of $c_1$ and $c_2$, we conclude that $(2+c_3)^2-(c_1^2+c_3^2)=0$. Thus, the exponent of the exponential function in \eqref{Shen_class} is nod defined. If we choose $a=1$, the same remark is valid for Asanov's class \eqref{Asanov_class}. However, we managed to obtain this class  by using Maple program and the metrizability tools.

\bigskip

$\bullet$
The fourth class is a general simple class of all non-regular $(\alpha,\beta)$-metrics which are Landsbergian and non-Berwaldian. Moreover, it is  equivalent to the class \eqref{Shen_class}. Comparing this class with Shen's one,  we have the constants $c_1=p$ and $c_3=q$. At these values of $c_1$ and $c_2$, we get $(2+c_3)^2-(c_1^2+c_3^2)=-(p^2-4q-4)$. For this equivalence, one may use   the property $\arctan z=-i \,\text{arctanh}\, z$, here $i=\sqrt{-1}$ is the complex number. It should be noted that if one compares the first class with the fourth class, we get that $p=2a$ and $q=a^2-1$. Therefore,  $p^2-4q-4=0$ which shows that the Finsler function in the fourth class is not defined, but by the help of Maple program and the metrizability criteria we obtained the solution.  It is worthy to mention  that the specific deformation of a geodesic  spray of Landsberg metric that we discussed at the beginning of the previous section is   changing  the constants $p$ and $q$ in specific way. In each class, we have a very special values of these constants.  Moreover,  the fourth class generalizes the second and third classes; indeed

 (i) In the second class,   we have the constants $p=2a$ and $q=a^2-2$. At these values of $p$ and $q$, we get $p^2-4q-4=4$. Using the fact that
 $$\text{arctanh}\ z=\frac{1}{2}\ln\left(\frac{1+z}{1-z}\right) $$
 we get the following equality
 $$e^{\frac{2p}{\sqrt{p^2-4q-4}}\, \text{arctanh}\left(\frac{p\beta+2\sqrt{\alpha^2-\beta^2}}{\beta\sqrt{p^2-4q-4}}\right)}=\left(-\frac{(a+1)\beta+\sqrt{\alpha^2-\beta^2}}{(a-1)\beta+\sqrt{\alpha^2-\beta^2}}\right)^{a}. $$
 In addition, by   factorization we have
 $$\alpha^2+2a\beta\sqrt{\alpha^2-\beta^2}+(a^2-2)\beta^2=\left((a+1)\beta+\sqrt{\alpha^2-\beta^2}\right)\left((a-1)\beta+\sqrt{\alpha^2-\beta^2}\right).$$
According to the above facts, one can conclude that the two energy functions associated with the second and fourth classes are equal,  except,   whenever $a$ is an odd (number or root) then  the two energy functions are equal up to a minus sign.

 \medskip

(ii)  In the third class,  we have the constants $p=\frac{3a}{2}$ and $q=\frac{a^2-2}{2}$. At these values of $p$ and $q$, we get $p^2-4q-4=\frac{1}{4}a^2$. By the same way, as we just discussed above,  the two energy functions associated with the third and fourth classes are equal, for any $a$, up to the factor $-4$. It is worthy to mention that the third class is one of the simplest classes of Landsberg metrics which are not Berwaldian.

\bigskip

$\bullet$ It is to  be noted that all the classes  established in this paper can be considered for any $\alpha$ and $\beta$ satisfying the conditions mentioned in the class \eqref{Shen_class}. We specified, in this paper, $\alpha$ and $\beta$ because the special choice of the coefficients of the sprays that  are proposed. For this reason,  we repeated the sentence \lq\lq with out loss of generality\rq\rq to confirm this fact.

\bigskip

$\bullet$ It should be noted that all the examples and classes studied in this paper satisfy the $\sigma$T-condition. Indeed,  by making use of \cite{Elgendi-ST_condition}, the metrics of $(\alpha,\beta)$-type satisfying the $\sigma$T-condition if and only if they are given by the class \eqref{Shen_class}. Assuming that   $\sigma_h=b_h$, taking into account the fact that $b_1=f(x1)\neq 0$, $b_2=...=b_n=0$,  we conclude that
$$T^1_{ijk}=0.$$
Which confirm that the T-tensor of a non-Berwaldian  Landsberg metric  has its specific  property.


\section*{Acknowledgment}
I would like to  express my deep gratitude to Professor  Szilasi J\'{o}zsef (University of Debrecen) for his valuable discussions on the topic of the paper and for his pushing to me to work on the unicorn's problem.

\providecommand{\bysame}{\leavevmode\hbox
to3em{\hrulefill}\thinspace}
\providecommand{\MR}{\relax\ifhmode\unskip\space\fi MR }
\providecommand{\MRhref}[2]{%
  \href{http://www.ams.org/mathscinet-getitem?mr=#1}{#2}
} \providecommand{\href}[2]{#2}

\end{document}